\documentclass[a4paper]{article}
\usepackage{amssymb,latexsym}
\usepackage[utf8]{inputenc}
\usepackage{amsmath,amsthm}
\usepackage{amsfonts,mathrsfs}
\usepackage{cleveref}
\usepackage{enumerate,units}
\usepackage{url}

\theoremstyle{plain}
\newtheorem{theorem}{Theorem}[subsection]
\newtheorem{corollary}[theorem]{Corollary}
\newtheorem{lemma}[theorem]{Lemma}
\newtheorem{proposition}[theorem]{Proposition}
\newtheorem*{theorem*}{Theorem}
\newtheorem*{proposition*}{Proposition}
\newtheorem*{propaux}{Proposition \theoremauxnum}
\gdef\theoremauxnum{1}

\newenvironment{propff}[2][]{%
  \def\theoremauxnum{\ref{#2}}
  \begin{propaux}[#1]
}{%
  \end{propaux}
}

\newtheorem*{theoremaux}{Theorem \theoremauxnum}
\gdef\theoremauxnum{1}

\newenvironment{theoremff}[2][]{%
  \def\theoremauxnum{\ref{#2}}
  \begin{theoremaux}[#1]
}{%
  \end{theoremaux}
}

\newtheorem*{exampaux}{Proposition \theoremauxnum}
\gdef\theoremauxnum{1}

\newtheorem*{claim}{Claim}

\theoremstyle{definition}
\newtheorem{definition}[theorem]{Definition}
\newtheorem*{definition*}{Definition}
\newtheorem{example}[theorem]{Example}

\theoremstyle{remark}
\newtheorem*{remark}{Remark}

\newtheorem*{question}{Question}

\numberwithin{equation}{section}

\newcommand{\forkindep}[1][]{%
  \mathrel{
    \mathop{
      \vcenter{
        \hbox{\oalign{\noalign{\kern-.3ex}\hfil$\vert$\hfil\cr
              \noalign{\kern-.7ex}
              $\smile$\cr\noalign{\kern-.3ex}}}
      }
    }\displaylimits_{#1}
  }
}

\begin{document}
\title{Semigroups in Stable Structures}
\author{Yatir Halevi\footnote{The research leading to these results has received funding from the European Research Council under the European Union’s Seventh Framework Programme (FP7/2007-2013)/ERC Grant Agreement No. 291111.}}
\date{}

\maketitle
\begin{abstract}
Assume $G$ is a definable group in a stable structure $M$. Newelski showed that the semigroup $S_G(M)$ of complete types concentrated on $G$ is an inverse limit of the $\infty$-definable (in $M^{eq}$) semigroups $S_{G,\Delta}(M)$. He also shows that it is strongly $\pi$-regular: for every $p\in S_{G,\Delta}(M)$ there exists $n\in\mathbb{N}$ such that $p^n$ is in a subgroup of $S_{G,\Delta}(M)$. We show that $S_{G,\Delta}(M)$ is in fact an intersection of definable semigroups, so $S_G(M)$ is an inverse limit of definable semigroups and that the latter property is enjoyed by all $\infty$-definable semigroups in stable structures.
\end{abstract}

\section{Introduction}
A Semigroup is a set together with an associative binary operation. Although the study of semigroups stems in the start of the 20th century not much attention has been given to semigroups in stable structures. One of the only facts known about them is
\begin{proposition*}{\cite{unidim}}
A stable semigroup with left and right cancellation, or with left cancellation and right identity, is a group.
\end{proposition*}

Recently $\infty$-definable semigroups in stable structures made an appearance in a paper by Newelski \cite{Newelski-stable-groups}:

Let $G$ be a definable group inside a stable structure $M$.  Define $S_G(M)$ to be all the types of $S(M)$ which are concentrated on $G$. $S_G(M)$ may be given a structure of a semigroup by defining for  $p,q\in S_G(M)$:
\[p\cdot q=tp(a\cdot b/M),\]
where $a\models p, b\models q$ and $a\forkindep[M]b$.

Newelski gives an interpretation of $S_{G,\Delta}(M)$ (where $\Delta$ is a finite set of invariant formulae) as an $\infty$-definable set in $M^{eq}$ and thus $S_G(M)$ may be interpreted as an inverse limit of $\infty$-definable semigroups in $M^{eq}$.

As a result he shows that for every local type $p\in S_{G,\Delta}(M)$ there exists an $n\in\mathbb{N}$ such that $p^n$ is in a subgroup of $S_{G,\Delta}(M)$. In fact he shows that $p^n$ is equal to a translate of a $\Delta$-generic of a $\Delta$-definable connected subgroup of $G(M)$.

\begin{definition*}
A semigroup $S$ is called \emph{strongly $\pi$-regular} or an \emph{epigroup} if for all $a\in S$ there exists $n\in\mathbb{N}$ such that $a^n$ is in a subgroup of $S$.
\end{definition*}

\begin{question}
Is this property enjoyed by all $\infty$-definable semigroups in stable structures? 
\end{question}

Since we're dealing with  $\infty$-definable semigroups, remembering that every $\infty$-definable group in a stable structure is an intersection of definable groups, an analogous questions arises:
 \begin{question}
Is every $\infty$-definable semigroup in a stable structure an intersection of definable ones? Is $S_{G,\Delta}(M)$ an intersection of definable semigroups?
\end{question}

In this paper we answer both theses questions. 

It is a classical result about affine algebraic semigroups that they are strongly $\pi$-regular. Recently Brion and Renner \cite{RenBrio} proved that this is true for all Algebraic Semigroups. In fact, we'll show that 
\begin{propff}{T:inf-def-is-spr}
Let $S$ be an $\infty$-definable semigroup inside a stable structure. Then $S$ is strongly $\pi$-regular.
\end{propff}
At least in the definable case, this is a direct consequence of stability, the general case is not harder but a bit more technical.

One can ask if what happens in $S_{G,\Delta}(M)$ is true in general $\infty$-definable semigroups. That is, is every element a power away from a translation of an idempotent. However, this is already is not true in $M_2(\mathbb{C})$.

As for the second question, in Section \ref{sec:S_G(M)} we show that $S_{G,\Delta}(M)$ is an intersection of definable semigroups. In fact,

\begin{theoremff}{T:invlim}
$S_G(M)$ is an inverse limit of definable semigroups in $M^{eq}$.
\end{theoremff}

Unfortunately not all $\infty$-definable semigroups are an intersection of definable ones.

Milliet showed that every $\infty$-definable semigroup inside a small structure is an intersection of definable semigroups \cite{on_enveloping}. In particular this is true for $\omega$-stable structures, and so for instance in $ACF$. Already in the superstable case this is not true in general, see \Cref{E:counter-exam}.

However there are some classes of semigroups in which this does hold. We recall some basic definitions from semigroup theory we'll need. See \Cref{ss:semigroups} for more information.
\begin{definition*}
\begin{enumerate}
\item An element $e\in S$ in a semigroup $S$ is an \emph{idempotent} if $e^2=e$.
\item A semigroup $S$ is called an \emph{inverse semigroup} if for every $a\in S$ there exists a unique $a^{-1}\in S$ such that 
\[aa^{-1}a=a,\quad a^{-1}aa^{-1}=a^{-1}.\]
\item A \emph{Clifford semigroup} is an inverse semigroup in which the idempotents are central. A \emph{surjective Clifford monoid} is a Clifford monoid in which for every $a\in S$ there exists $g\in G$ and idempotent $e$ such that $a=ge$, where $G$ is the unit group of $S$.
\end{enumerate}
\end{definition*}

These kinds of semigroups do arise in the context of $S_G(M)$.
It is probably folklore, but one may show (see \Cref{ss:S_G inverse}) that if $G$ is $1$-based then $S_G(M)$ is an inverse monoid. In \Cref{ss:S_G inverse} we give a condition on $G$ for $S_G(M)$ to be Clifford.

\begin{theoremff}{T:Cliff_with_surj}
Let $S$ be an $\infty$-definable surjective Clifford monoid in a stable structure. Then $S$ is contained in a definable monoid, extending the multiplication on $S$. This monoid is also a surjective Clifford monoid.
\end{theoremff}
As a result from the proof, every such monoid is an intersection of definable ones.

In the process of proving the above Theorem we show two results which might be interesting in their own right.

Since $\infty$-definable semigroups in stable structures are s$\pi$r, one may define a partial order on them given by 
\[a\leq b \Leftrightarrow a=be=fb \text{ for some } e,f\in E(S^1),\]
where $S^1$ is $S\cup \{1\}$ where we define $1$ to be the identity element.
If for every $a,b\in S$, $a\cdot b\leq a,b$ one may show that there exists $n\in\mathbb{N}$ such that every product of $n+1$ elements is already a product of $n$ of them (\Cref{P:in negative_order_n+1_is_n}). As a result any such semigroup is an intersection of definable ones. In particular,

\begin{propff}{P:idemp_is_inside_definable}
Let $E$ be an $\infty$-definable commutative idempotent semigroup inside a stable structure, then $E$ is contained in a definable commutative idempotent semigroup. Furthermore, it is an intersection of definable ones.
\end{propff}

\section{Preliminaries}\label{S:Preliminaries}

\subsection{Notations}
We fix some notations.
We'll usually not distinguish between singletons and sequences thus we may write $a\in M$ and actually mean $a=(a_1,\dots,a_n)\in M^n$, unless a distinguishment is necessary. $A,B,C,\dots$ will denote parameter sets and $M,N,\dots$ will denote models. When talking specifically about semigroups, monoids and groups (either definable, $\infty$-definable or models) we'll denote them by $S$, $M$ and $G$, respectively. We use juxtaposition $ab$ for concatenation of sequences, or $AB$ for $A\cup B$ if dealing with sets. That being said, since we will be dealing with semigroups, when there is a chance of confusion we'll try to differentiate between the concatenation $ab$ and the semigroup multiplication $ab$ by denoting the latter by $a\cdot b$.

\subsection{Semigroups}\label{ss:semigroups}
Clifford and Preston \cite{Clifford1,Clifford2} is still a very good reference for the theory of semigroups, but Higgins \cite{Higgins} and Howie \cite{Howie} are much more recent sources.

A set $S$ with an associative binary operation is called a \emph{semigroup}.\\
An element $e\in S$ is an \emph{idempotent} if $e^2=e$. We'll denote by $E(S)$ the subset of all idempotents of $S$.\\ 
By a subgroup of $S$ we mean a subsemigroup $G\subseteq S$ such that there exists an idempotent $e\in G$ such that $(G,\cdot)$ is a group with neutral element $e$.\\
$S$ is \emph{strongly $\pi$-regular} (s$\pi$r) if for each $a\in S$ there exits $n>0$ such that $a^n$ lies in a subgroup of $S$.
\begin{remark}
These type of semigroups are also known as \emph{epigroups} and their elements as \emph{group-bound}.
\end{remark}

A semigroup with an identity element is called a \emph{monoid}. Notice that any semigroup can be extended to a monoid by artificially adding an identity element. We'll denote it by $S^1$. If $S$ is a monoid we'll denote by $G(S)$ its subgroup of invertible elements.\\
Given two semigroups $S,S^\prime$, a \emph{homomorphism of semigroups} is map \[\varphi :S\to S^\prime\] such that $\varphi(xy)=\varphi(x)\varphi(y)$ for all $x,y\in S$. If $S,S^\prime$ are monoids, then we say $\varphi$ is a \emph{homomorphism of monoids} if in addition $\varphi(1_S)=1_{S^\prime}$.\\

\begin{definition}
The \emph{natural partial order} on $E(S)$ is defined by \[e\leq f \Leftrightarrow ef=fe=e.\]
\end{definition}

\begin{proposition}\cite[Section 1.7]{Clifford1}\label{P:Max.subgroups}
For every $e,f\in E(S)$ we have the following:
\begin{enumerate}
\item $eSe$ is a subsemigroup of $S$. In fact, it is a monoid with identity element $e$;
\item $eSe\subseteq fSf \Leftrightarrow e\leq f;$
\item Every maximal subgroup of $S$ is of the form $G(eSe)$ (the unit group of $eSe$) for $e\in E(S)$;
\item If $e\neq f$ then $G(eSe)\cap G(fSf)=\emptyset$.
\end{enumerate}
\end{proposition}

There are various ways to extend the partial order on the idempotents to a partial order on the entire semigroup. See \cite[Section 1.4]{Higgins} for a discussion about them. We'll use the \emph{natural partial order} on $S$. It has various equivalent definitions, we present the one given in \cite[Proposition 1.4.3]{Higgins}.

\begin{definition}\label{D:general_def_of_order}
The relation 
\[a\leq b \Leftrightarrow a=xb=by,xa=a \text{ for some } x,y\in S^1\]
is called the \emph{natural partial order} on $S$.
\end{definition}
Notice that this extends the partial order on $E(S)$. If $S$ is s$\pi$r this partial order takes a more elegant form:

\begin{proposition}\cite[Corollary 1.4.6]{Higgins}
On s$\pi$r semigroups there is a natural partial order extending the order on $E(S)$:
\[a\leq b \Leftrightarrow a=be=fb \text{ for some } e,f\in E(S^1).\]
\end{proposition}

\subsubsection{Clifford and Inverse Semigroups}\label{subsec:Clifford}
\begin{definition}
A semigroup $S$ is called $regular$ if for every $a\in S$ there exists at least one element $b\in S$ such that 
\[aba=a,\quad bab=b.\]
Such an element $b$ is a called a \emph{pseudo-inverse} of $a$
\end{definition}

\begin{definition}
A semigroup $S$ is called an \emph{inverse semigroup} if for every $a\in S$ there exists a unique $a^{-1}\in S$ such that 
\[aa^{-1}a=a,\quad a^{-1}aa^{-1}=a^{-1}.\]
\end{definition}

Basic facts about inverse semigroups:
\begin{proposition}\cite[Section V.1, Theorem 1.2 and Proposition 1.4]{Howie}
Let $S$ be an inverse semigroup.
\begin{enumerate}
\item For every $a,b\in S$, $\left( a^{-1}\right)^{-1}=a$ and $\left( ab\right)^{-1}=b^{-1}a^{-1}.$
\item For every $a\in S$, $aa^{-1}$ and $a^{-1}a$ are idempotents.
\item The idempotents commute. Thus $E(S)$ is a commutative subsemigroup, and hence a semilattice.
\end{enumerate}
\end{proposition}

The basic example for inverse semigroups is the set, $\mathscr{I}(X)$, of partial one-to-one mappings for a set $X$, that means that the domain is a (possibly empty) subset of $X$. The composition of two "incompatible" mappings will be the empty mapping. The first surprising fact is that this is in fact an \emph{inverse} semigroup, but one can say even more (a generalization of Cayley's theorem for groups):

\begin{theorem}[The Vagner-Preston Representation Theorem]\cite[Section V.1, Theorem 1.10]{Howie}
If $S$ is an inverse semigroup then there exists a set $X$ and a monomorphism $\phi: S\to \mathscr{I}(X)$.
\end{theorem}

If $S$ is an inverse semigroup, the partial order on $S$ gets the following form: $a\leq b$ if there exists $e\in E(S)$ such that $a=eb$. 

\begin{proposition}
\begin{enumerate}
\item $\leq$ is a partial order relation.
\item If $a,b,c\in S$ such that $a\leq b$ then $ac\leq bc$ and $ca\leq cb$. Futhermore, $a^{-1}\leq b^{-1}$.
\end{enumerate}
\end{proposition}

\begin{definition}
A \emph{Clifford semigroup} is an inverse semigroup in which the idempotents are central. 
\end{definition}

\begin{remark}
Different sources give different, but equivalent, definitions of a Clifford semigroup. For instance Howie defines a Clifford semigroup to be a regular semigroup $S$ in which the idempotents are central \cite[Section IV.2]{Howie}. One may show that $S$ is an inverse semigroup if and only if it is regular and the idempotents commute \cite[Section V.1, Theorem 1.2]{Howie}, so the definitions coincide.
\end{remark}

The following is well known, but we'll add a proof instead of adding another source.
\begin{proposition}
$S$ is a Clifford semigroup if and only if it is an inverse semigroup and $aa^{-1}=a^{-1}a$ for all $a\in S$. 
\end{proposition}
\begin{proof}
Assume $S$ is a Clifford semigroup and let $a\in S$. Since $aa^{-1}$ and $a^{-1}a$ are idempotents and central, 
\[aa^{-1}=a(a^{-1}a)a^{-1}=(a^{-1}a)aa^{-1}=a^{-1}a(aa^{-1})=a^{-1}(aa^{-1})a=a^{-1}a.\]
Conversely, we must show that the idempotents are central. For $a\in S$ and $e\in E(S)$, we'll show that $ea=(ea)(a^{-1}e)(ea)=(ea)(ea^{-1})(ea)$ and thus by the uniqueness of the pseudo-inverses, $ea=ae$. By our assumption
\[(ea)(ea^{-1})(ea)=eae(a^{-1}e)(ea)=eae(ea)(a^{-1}e)=eaeaa^{-1}e\]
Again, $(ea)(a^{-1}e)=(a^{-1}e)(ea)$ so
\[=eaa^{-1}eea=eaa^{-1}ea,\]
and by the commutativity of the idempotents ($e$ and $aa^{-1}$),
\[=eaa^{-1}a=ea.\]
\end{proof}

\begin{definition}\cite[Chapter IV]{Howie}
A semigroup $S$ is said to be a \emph{strong semilattice of semigroups} if there exists a semilattice $Y$, disjoint subsemigroups $\{S_{\alpha}:\alpha\in Y\}$ and homomorphisms $\{\phi_{\alpha,\beta}:S_\alpha\to S_\beta:\alpha,\beta\in Y, \alpha\geq \beta\}$ such that
\begin{enumerate}
\item $S=\bigcup_\alpha S_\alpha$.
\item $\phi_{\alpha,\alpha}$ is the identity.
\item For every $\alpha\geq \beta\geq\gamma$ in $Y$, $\phi_{\beta,\gamma}\phi_{\alpha,\beta}=\phi_{\alpha,\gamma}$.
\end{enumerate}
\end{definition}

\begin{theorem}\cite[Section IV.2, Theorem 2.1]{Howie}
$S$ is a Clifford semigroup if and only if it is a strong semilattice of groups. The semilattice is $E(S)$, the disjoint groups are \[\{G_e=G(eSe): e\in E(S)\}\] the maximal subgroups of $S$ and the homomorphism $\phi_{e,ef}$ is given by multiplication by $f$.
\end{theorem}

\section{$\infty$-definable Semigroups and Monoids}\label{S:big-wedge-defin}

Let $S$ be an $\infty$-definable semigroup in a stable structure. Assume that $S$ is defined by \[\bigwedge_i \varphi_i(x).\]

\begin{remark}
We assume that $S$ is defined over $\emptyset$ just for notational convenience. Moreover we assume that the $\varphi_i$s are closed under finite conjunctions.
\end{remark}

\subsection{Strongly $\pi$-regular}

Our goal is to prove that an $\infty$-definable semigroup inside a stable structure is s$\pi$r. To better understand what's going on we start with an easier case:

\begin{definition}
A \emph{stable semigroup} is a stable structure $S$ such that there is a definable binary function $\cdot$ which makes $(S,\cdot)$ into a semigroup. 
\end{definition}

The following was already noticed in \cite{LoseyHans} for semigroups with chain conditions, but we give it in a "stable semigroup" setting.

\begin{proposition}\label{P:idem} 
Any stable semigroup has an idempotent.
\end{proposition}
\begin{proof}
Let $a\in S$ and let \[\theta(x,y)= \exists u (u\cdot a=a\cdot u\wedge u\cdot x=y).\] Obviously, $S\models \theta(a^{3^m},a^{3^n})$ for $m<n$. $S$ is stable hence $\theta$ doesn't have the order property. Thus there exists $m<n$ such that $S\models \theta(a^{3^n},a^{3^m})$. Let $C\in S$ be such that $C\cdot a^{3^n}=a^{3^m}$ and commutes with $a$.

Since $3^n> 2\cdot3^m$ then multiplying by $a^{3^n-2\cdot 3^m}$ yields $Ca^{2(3^n-3^m)}=a^{3^n-3^m}$. Notice that since $C$ commutes with $a$, $Ca^{3^n-3^m}$ is an idempotent.
\end{proof}

\begin{proposition}\label{P:s pi r}
Any stable semigroup is $s\pi r$.
\end{proposition}
\begin{proof}
Let $a\in S$. From the proof of \Cref{P:idem} there exists $C\in S$ that commutes with $a$ and $n>0$ such that $Ca^{2n}=a^n$. Set $e:=Ca^n$. Indeed, $a^n=e\cdot a^n\cdot e$ and $a^n\cdot eCe=e$.
\end{proof}
\begin{remark}
Given $a\in S$, there exists a unique idempotent $e=e_a\in S$ such that $a^n$ belongs to the unit group of $eSe$ for some $n>0$. Indeed, for two idempotents $e\neq f$ the unit groups of $eSe$ and $fSf$ are disjoint (Proposition \ref{P:Max.subgroups}). 
\end{remark}

Furthermore, we have
\begin{lemma}\cite{Munn}\label{L:higher-power-in-max-subgrp}
Let $S$ be a semigroup and $x\in S$. If for some $n$, $x^n$ lies in a subgroup of $S$ with identity $e$ then $x^m$ lies in the unit group of $eSe$ for all $m\geq n$.
\end{lemma}

\begin{corollary}\label{C:universal n}
There exists $n>0$ (depending only on $S$) such that for all $a\in S$, $a^n$ belongs to the unit group of $e_aSe_a$.
\end{corollary}
\begin{proof}
Let $\phi_i(x)$ be the formula '$x^i\in \text{ the unit group of } e_xSe_x$'. $\bigcup_i[\phi(x)]=S_1(S)$, since every elementary extension of $S$ is also stable and hence $s\pi r$. By compactness there exist $n_1,\dots,n_k>0$ such that $S_1(S)=[\phi_{n_1}\vee\dots\vee \phi_{n_k}]$. $n=n_1\cdots n_k$ is our desired integer.
\end{proof}

We return to the general case of $S$ being an $\infty$-definable  semigroup inside a stable structure. The following is an easy consequence of stability:

\begin{proposition} \label{P:chain of idempotents}
Every chain of idempotents in $S$, with respect to the partial order on them, is finite and uniformly bounded.
\end{proposition}

Our goal is to show that for every $a\in S$ there exists an idempotent $e\in S$ and $n\in\mathbb{N}$ such that $a^n$ is in the unit group of $eSe$.

We'll want to assume that $S$ is a conjunction of countably many formulae. For that we'll need to make some observations, the following is well known but we add a proof for completion,
\begin{lemma}
Let $S$ be an $\infty$-definable semigroup. Then there exist $\infty$-definable semigroups $H_i$ such that each $H_i$ is defined by at most a countable set of formulae, and $S=\bigcap H_i$.
\end{lemma}
\begin{proof}
Let $S=\bigwedge_{i\in I} \varphi_i$ and assume that the $\varphi_i$s are closed under finite conjunctions. By compactness we may assume that for all $i$ and $x,y,z$
$$\varphi_i(x)\wedge\varphi_i(y)\wedge\varphi_i(z)\to (xy)z=x(yz).$$
Let $i^0\in I$. By compactness there exists $i^0_1\in I$ such that for all $x,y$:
$$\varphi_{i^0_1}(x)\wedge \varphi_{i^0_1}(y)\to \varphi_{i^0}(xy).$$
Thus construct a sequence $i^0,i^0_1,i^0_2,\dots$ and define 
$$H_{i^0}=\bigwedge_j \varphi_{i^0_j}.$$
This is indeed a semigroup and 
$$S=\bigcap_{i\in I} H_i.$$
\end{proof}

The following is also well known,

\begin{proposition}{\cite{unidim}}
An $\infty$-definable semigroup in a stable structure with left and right cancellation, or with left cancellation and right identity, is a group.
\end{proposition}

As a consequence,
\begin{lemma}\label{L:max_subgrps_are_inf_definable}
Let $S$ be an $\infty$-definable semigroup and $G_e\subseteq S$ a maximal subgroup (with idempotent $e\in E(S)$). $G_e$ is relatively definable in $S$.
\end{lemma}
\begin{proof}
Let $S=\bigwedge_i \varphi_i(x)$. By compactness, there exists a definable set $S\subseteq S_0$ such that for all $x,y,z\in S_0$ \[x(yz)=(xy)z.\]
Let $G_e(x)$ be 
\[\bigwedge_i\varphi_i(x)\wedge (xe=ex=x)\wedge \bigwedge_i(\exists y\in S_0)(\varphi_i(y)\wedge ye=ey=y\wedge yx=xy=e).\]
This $\infty$-formula defines the maximal subgroup $G_e$. Indeed if $a\models G_e(x)$ and $b,b^\prime\in S_0$ are such that 
\[\varphi_i(b)\wedge be=eb=b\wedge ba=ab=e\]
and
\[\varphi_j(b^\prime)\wedge b^\prime e=eb^\prime=b^\prime\wedge b^\prime a=ab^\prime=e,\]
then 
\[b^\prime=b^\prime e=b^\prime(ab)=(b^\prime a)b=eb=b.\]
Hence there exists an inverse of $a$ in $S$.

Let $G_e\subseteq G_0$ be a definable group containing $G_e$ (see \cite{unidim}). $G_0\cap S$ is an $\infty$-definable subsemigroup of $S$ with cancellation, hence a subgroup. It is thus contained in the maximal subgroup $G_e$ and so equal to it.
\end{proof}

\begin{lemma}
Let $S$ be an $\infty$-definable semigroup and $S\subseteq S_1$ an $\infty$-definable semigroup containing it. If $S_1$ is s$\pi$r then so is $S$.
\end{lemma}
\begin{proof}
Let $a\in S$ and let $a^n\in G_e\subseteq S_1$ where $G_e$ is a maximal subgroup of $S_1$. Thus $a^n\in G_e\cap S$. Since $G_e\cap S$ is an $\infty$-definable subsemigroup of $S$ with cancellation, it is a subgroup.
\end{proof}

We may, thus, assume that $S$ is the conjunction of countably many formulae.

Furthermore, we may, and will, assume that $S$ is commutative. Indeed, let $a\in S$. By compactness we may find a definable set $S\subseteq S_0$ such that for all $x,y,z\in S_0$:
\[x(yz)=(xy)z.\]
Define $D_1=\{x\in S_0: xa=ax\}$ and then \[D_2=\{x\in S: (\forall c\in D_1)\: xc=cx\}.\]
$D_2$ is an $\infty$-definable commutative subsemigroup of $S$ with $a\in D_2$.

\begin{lemma}\label{L:technical-stuff-for-spr}
There exist definable sets $S_i$ such that $S=\bigcap S_i$, the multiplication on $S_i$ is commutative and that for all $1<i$ there exists $C_i\in S_i$ and $n_i,m_i\in \mathbb{N}$ such that
\begin{enumerate}
\item $n_i>2m_i$;
\item $e_i:=C_ia^{n_i-m_i}$ is an idempotent;

and furthermore for all $1<j\leq i$:
\item $n_j-m_j\leq n_i-m_i$;
\item $e_je_i=e_i$;
\item $e_ia^{n_i-m_i}=a^{n_i-m_i}$.

\end{enumerate}
\end{lemma}
\begin{proof}
By compactness we may assume that $S=\bigcap S_i$, where \[S_0\supseteq S_1\supseteq S_2\supseteq \dots\]
are definable sets such that for all $i>1$ we are allowed to multiply associatively and commutatively $\leq 20$ elements of $S_i$ and get an element of $S_{i-1}$.

Let $i>1$ and let $\theta(x,y)$ be 
\[\exists u\in S_i \: ux=y.\] 
Obviously, $\models \theta(a^{3^k},a^{3^l})$ for $k<l$. By stability $\theta$ doesn't have the order property. Thus there exist $k<l$ such that $\models \theta(a^{3^l},a^{3^k})$. Let $C_i\in S_i$ be such that $C_ia^{3^l}=a^{3^k}$.

Since $l>k$ we have $3^l> 2\cdot 3^k$ (this gives 1). Let $n_i=3^l$ and $m_i=3^k$. $e_i:=C_ia^{n_i-m_i}\in S_{i-1}$ is an idempotent (this gives 2), for that first notice that:
\[C_ia^{2n_i-2m_i}=C_ia^{n_i}a^{n_i-2m_i}=a^{m_i}a^{n_i-2m_i}=a^{n_i-m_i}.\]
Hence,
\[(C_ia^{n_i-m_i})(C_ia^{n_i-m_i})=C_i^2a^{2n_i-2m_i}=C_ia^{n_i-m_i}.\]

We may take $n_i-m_i$ to be minimal, but then since $S_i\subseteq S_j$ for $j<i$ we have $n_j-m_j\leq n_i-m_i$ (this gives 3).

As for 4, if $1<j<i$:
\[e_ie_j=C_ia^{n_i-m_i}C_ja^{n_j-m_j}=C_ia^{n_i-m_i+m_j}C_ja^{n_j-2m_j}\]
but $n_i-m_i+m_j\geq n_j$, so
\[C_ia^{n_i-m_i+m_j-n_j}C_ja^{2n_j-2m_j}=C_ia^{n_i-m_i+m_j-n_j}a^{n_j-m_j}=e_i.\]
5 follows quite similarly to what we've done.
\end{proof}

\begin{proposition}\label{T:inf-def-is-spr}
Let $S$ be an $\infty$-definable semigroup inside a stable structure. Then $S$ is strongly $\pi$-regular.
\end{proposition}
\begin{proof}
Let $a\in S$. For all $i>1$ let $S_i, C_i, n_i$ and $m_i$ be as in Lemma \ref{L:technical-stuff-for-spr}. Set $k_i=n_i-m_i$ and
\[e_{i-1}=C_ia^{k_i} \text{ and } \beta_{i-1}=e_iC_ie_i,\]
notice that these are both elements of $S_{i-1}$ (explaining the  sub-index). 

By Lemma \ref{L:technical-stuff-for-spr}{(4)} we get a descending sequence of idempotents \[e_1\geq e_2\geq \dots,\] with respect to the partial order on the idempotents. By stability it must stabilize. 
Thus we may assume that $e:=e_1=e_2=\dots$ and is an element of $S$.

Moreover, for all $i>1$
\[\beta_1=\beta_1\cdot e=\beta_1a^{k_{i+1}}\cdot \beta_i=e\cdot a^{k_{i+1}-k_2}\beta_i.\]
So \[\beta_1=e\cdot a^{k_{i+1}-k_2}eC_{i+1}e\]
which is a product of $\leq 20$ elements of $S_{i+1}$ and thus $\in S_i$. Also $\beta_1\in S$.
In conclusion, by setting $k:=k_2$ and $\beta:=\beta_1$,
\[a^ke=ea^k=a^k, \text{ } a^k\beta=\beta a^k=e \text{ and } \beta e=e\beta=\beta.\]
So $a^k$ is in the unit group of $eSe$.
\end{proof}

\begin{corollary}
There exists an $n\in\mathbb{N}$ such that for all $a\in S$, $a^n$ is an element of a subgroup of $S$.
\end{corollary}
\begin{proof}
Compactness.
\end{proof}

\begin{corollary}
$S$ has an idempotent.
\end{corollary}

\begin{remark}
In the notations of \Cref{sec:S_G(M)}, Newelski showed in \cite{Newelski-stable-groups} that $S_{G,\Delta}(M)$ is an $\infty$-definable semigroup in $M^{eq}$ and that it is s$\pi$r. \Cref{T:inf-def-is-spr}, thus, gives another proof.
\end{remark}

\subsection{A Counter Example}
It is known that every $\infty$-definable group inside a stable structure is an intersection of definable ones. It would be even better if every such semigroup were an intersection of definable semigroups. Milliet  showed that every $\infty$-definable semigroup inside a small structure is an intersection of definable semigroups \cite{on_enveloping}. In particular this is true for $\omega$-stable structures. So, for instance, any $\infty$-definable subsemigroup of $M_n(k)$ for $k\models ACF$ is an intersection of definable semigroups.
Unfortunately, this is not true already in the superstable case, as the following example will show.

\begin{example}\label{E:counter-exam}
Pillay and Poizat  give an example of an $\infty$-definable equivalence relation which is not an intersection of definable ones \cite{Pillay-Poizat}. This will give us our desired semigroup structure.

Consider the theory of a model which consists of universe $\mathbb{Q}$ (the rationals) with the unary predicates:
\[U_a=\{x\in\mathbb{Q}: x\leq a\}\] for $a\in\mathbb{Q}$.
The equivalence relation, $E$, is defined by
\[\bigwedge_{a<b} ((U_a(x)\to U_b(y))\wedge (U_a(y)\to U_b(x)).\]
It is an equivalence relation and in particular a preorder (reflexive and transitive).
Notice that it also follows that $E$ can't be an intersection of definable preorders. For if $E=\bigwedge R_i$ (for preorders $R_i$) then we also have \[E=\bigwedge (R_i\wedge\overline{R}_i)\]
where $x\overline{R}_iy=yR_ix$ (since $E$ is symmetric). But $R_i\wedge\overline{R}_i$ is a definable equivalence relation (the symmetric closure) hence trivial. So the $R_i$ are trivial.

Milliet  showed that in an arbitrary structure, every $\infty$-definable semigroup is an intersection of definable semigroups if and only if this is true for all $\infty$-definable preorders \cite{on_enveloping}. As a consequence in the above structure we can define an $\infty$-definable semigroup which will serve as a counter-example. Specifically, it will be the following semigroup:

If the preorder is on a set $X$, add a new element $0$ and add $0R0$ to the preorder. Define a semigroup multiplication on $R$:
\begin{equation*}
(a,b)\cdot (c,d) =
\begin{cases}
(a,d) & \text{if }b=c,\\
(0,0) & \text{else } 
\end{cases}
\end{equation*}

\begin{remark}
This example also shows that even "presumably well behaved" $\infty$-definable semigroups need not be an intersection of definable ones. In the example at hand the maximal subgroups are uniformly definable (each of them is finite) and the idempotents form a commutative semigroup.
\end{remark}
\end{example}

\subsection{Semigroups with Negative Partial Order}\label{Subsec:inf-def-with-neg-part}

We showed in \Cref{T:inf-def-is-spr} that every $\infty$-definable semigroup in a stable structure is strongly $\pi$-regular, hence the natural partial order on it has the following form:

For any $a,b\in S$,  $a\leq b$ if there exists $f,e\in E(S^1)$ such that $a=be=fb$.

\begin{remark}
Notice that this order generalizes the order on the idempotents.
\end{remark}

In a similar manner to what was done with the order of the idempotents, we have the following:

\begin{proposition}
$ $
\begin{enumerate}[(i)]
\item Every chain of elements with regard to the natural partial order is finite.
\item By compactness, the length of the chains is bounded.
\end{enumerate}
\end{proposition}

\begin{definition}
We'll say that a semigroup $S$ is \emph{negatively ordered} with respect to the partial order if \[a\cdot b\leq a,b\]
for all $a,b\in S$.
\end{definition}

\begin{example}
A commutative idempotent semigroup (a (inf)-semilattice) is negatively ordered. 
\end{example}

Negativily ordered semigroups were studied by Maia and Mitsch \cite{NegPart}. We'll only need the definition.

\begin{proposition}\label{P:in negative_order_n+1_is_n}
Let $S$ be a negatively ordered semigroup. Assume that the length of chains is bounded by $n$, then any product of $n+1$ elements is a product of $n$ of them.
\end{proposition}
\begin{proof}
Let $a_1\cdot\dotsc\cdot a_{n+1}\in S$. Since $S$ is negatively ordered, 
\[a_1\cdot \dotsc \cdot a_{n+1}\leq a_1\cdot\dotsc\cdot a_n\leq\dots\leq a_1\cdot a_2\leq a_1.\]
Since $n$ bounds the length of chains, we must have 
\[a_1\cdot\dotsc\cdot a_i=a_1\cdot\dotsc\cdot a_{i+1}\]
for a certain $1\leq i\leq n$.
\end{proof}

This property is enough for an $\infty$-definable semigroup to be contained inside a definable one.

\begin{proposition}\label{P:product_of_n+1_is_n_is_inside_def}
Let $S$ be an $\infty$-definable semigroup (in any structure). If every product of $n+1$ elements in $S$ is a product of $n$ of them, then $S$ is contained inside a definable semigroup. Moreover, $S$ is an intersection of definable semigroups.
\end{proposition}
\begin{proof}
Let $S\subseteq S_0$ be a definable set where the multiplication is defined. By compactness, there exists a definable subset $S\subseteq S_1\subseteq S_0$ such that 
\begin{itemize}
\item Any product of $\leq 3n$ elements of $S_1$ is an element of $S_0$;
\item Associativity holds for products of $\leq 3n$ elements of $S_1$;
\item Any product of $n+1$ elements of $S_1$ is already a product of $n$ of them.
\end{itemize}
Let \[S_1\subseteq S_2=\{x\in S_0: \exists y_1,\dots,y_n\in S_1\quad \bigvee_{i=1}^n x=y_1\cdot \dotsc\cdot y_i\}.\]
We claim that if $a\in S_1$ and $b\in S_2$ then $ab\in S_2$, indeed this follows from the properties of $S_1$. Define
\[S_3=\{x\in S_2:xS_2\subseteq S_2\}.\]
$S_3$ is our desired definable semigroup.

\end{proof}

As a consequence of these two Propositions, we have
\begin{proposition}
Every $\infty$-definable negatively ordered semigroup inside a stable structure is contained inside a definable semigroup. Furthermore, it is an intersection of definable semigroups.
\end{proposition}

Since every commutative idempotent semigroup is negatively ordered we have the following Corollary,
\begin{corollary}\label{P:idemp_is_inside_definable}
Let $E$ be an $\infty$-definable commutative idempotent semigroup inside a stable structure, then $E$ is contained in a definable commutative idempotent semigroup. Furthermore, it is an intersection of definable ones.
\end{corollary}
\begin{proof}
We only need to show that the definable semigroup containing $E$ can be made commutative idempotent. For that to we need to  demand that all the elements of $S_1$ (in the proof of \Cref{P:product_of_n+1_is_n_is_inside_def}) be idempotents and that they commute, but that can be satisfied by compactness.
\end{proof}

\subsection{Clifford Monoids}\label{Subsec:Cliff} 
We assume that $S$ is an $\infty$-definable Clifford semigroup (see \Cref{subsec:Clifford}) inside a stable structure.

The simplest case of Clifford semigroups, commutative idempotent semigroups (semilattices) were considered in \Cref{Subsec:inf-def-with-neg-part}.

Understanding the maximal subgroups of a semigroup is one of the first steps when one wishes to understand the semigroup itself. Lemma \ref{L:max_subgrps_are_inf_definable} is useful and will be used implicitly.

Recall that every Clifford semigroup is a strong semilattice of groups. Between each two maximal subgroups, $G_e$ and $G_{ef}$ there exists a homomorphism $\phi_{e,ef}$ given by multiplication by $f$.

\begin{definition}
By a \emph{surjective Clifford monoid} we mean a Clifford monoid $M$ such that for every $a\in M$ there exist $g\in G(M)$ and $e\in E(M)$ such that $a=ge$.

Surjectivity refers to the fact that these types of Clifford monoids are exactly the ones with $\phi_{e,ef}$ surjective.
\end{definition}

We restrict ourselves to $\infty$-definable surjective Clifford monoids. 
\begin{theorem}\label{T:Cliff_with_surj}
Let $M$ be an $\infty$-definable surjective Clifford monoid in a stable structure. Then $M$ is contained in a definable monoid, extending the multiplication on $M$. This monoid is also a surjective Clifford monoid. Furthermore, every such monoid is an intersection of definable surjective Clifford monoids.
\end{theorem}
\begin{proof}
Let $M\subseteq M_0$ be a definable set where the multiplication is defined.
By compactness, there exists a definable subset $M\subseteq M_1\subseteq M_0$
such that 
\begin{itemize}
\item Associativity holds for $\leq 6$ elements of $M_1$;
\item Any product of $\leq 6$ elements of $M_1$ is in $M_0$;
\item $1$ is a neutral element of $M_1$;
\item If $x$ and $y$ are elements of $M_1$ with $y$ an idempotent then $xy=yx$.
\end{itemize}
By the standard argument for stable groups, there exists a definable group \[G_1\subseteq G\subseteq M_1,\] where $G_1\subseteq M$ is the maximal subgroup of $M$ associated with the idempotent $1$.
By \Cref{P:idemp_is_inside_definable}, there exists a definable commutative idempotent semigroup $E(M)\subseteq E\subseteq M_1$.
Notice that for every $g\in G$ and $e\in E$, \[ge=eg.\]
Define \[M_2=\{m\in M_0:\exists g\in G,e\in E\quad m=ge\}.\]
$M_2$ is the desired monoid.

The furthermore is a standard corollary of the above proof.
\end{proof}

\begin{remark}\label{C:Cliff_is_inter}
As before, it follows from the proof that any such monoid is an intersection of definable surjective Clifford monoids.
\end{remark}

We don't have an argument for Clifford monoids which are not necessarily surjective. But we do have a proof for a certain kind of inverse monoids. We'll need this result in \Cref{sec:S_G(M)}.

\begin{theorem}\label{T:Needed_for_1_based}
Let $M$ be an $\infty$-definable monoid in a stable structure, such that
\begin{enumerate}
\item Its unit group $G$ is definable,
\item $E(M)$ is commutative, and 
\item for every $a\in M$ there exist $g\in G$ and $e\in E(M)$ such that \[a=ge.\]
\end{enumerate}
Then $M$ is contained in a definable monoid, extending the multiplication on $M$. This monoid also has these properties.
\end{theorem}
\begin{remark}
Incidentally $M$ is an inverse monoid (recall the definition from Section \ref{subsec:Clifford}). It is obviously regular and the pseudo-inverse is unique since the idempotents commute (see the preliminaries). Also, as before, every such monoid is an intersection of definable ones.
\end{remark}
\begin{proof}
Let $M\subseteq M_0$ be a definable set where the multiplication if defined and associative.
By compactness, there exists a definable subset $M\subseteq M_1\subseteq M_0$
such that 
\begin{itemize}
\item Associativity holds for $\leq 6$ elements of $M_1$;
\item Any product of $\leq 6$ elements of $M_1$ is in $M_0$;
\item $1$ is a neutral element of $M_1$;
\item If $x$ and $y$ are idempotents of $M_1$ then $xy=yx$.
\end{itemize}
By Proposition \ref{P:idemp_is_inside_definable}, there exists a definable commutative idempotent semigroup $E(M)\subseteq E\subseteq M_1$.

Let \[E_1=\{e\in E: \forall g\in G\: g^{-1}eg\in E\}.\]
$E_1$ is still a definable commutatve idempotent semigroup that contains $E(M)$. Moreover for every $e\in E_1$ and $g\in G$, \[g^{-1}eg\in E_1.\]
Define \[M_2=\{m\in M_0:\exists g\in G,e\in E_1\quad m=ge\}.\]
$M_2$ is the desired monoid. Indeed if $g,h\in G$ and $e,f\in E_1$ then there exist $h^\prime\in G$ and $e^\prime\in E_1$ such that 
\[eh=h^\prime e^\prime\]
thus \[ge\cdot hf=gh^\prime\cdot e^\prime f.\]
\end{proof}

\section{The space of types $S_G(M)$ on a definable group}\label{sec:S_G(M)}
Let $G$ be a definable group inside a stable structure $M$. Assume that $G$ is definable by a formula $G(x)$. Define 
$S_G(M)$ to be all the types of $S(M)$ which are on $G$.

\begin{definition}\label{D:of product}
Let $p,q\in S_G(M)$, define
\[p\cdot q=tp(a\cdot b/M),\]
where $a\models p, b\models q$ and $a\forkindep[M]b$.

\end{definition}
Notice that the above definition may also be stated in the following form:
\[U\in p\cdot q\Leftrightarrow  d_q(U)\in p.\]
where $U$ is a formula and $d_q(U):=\{g\in G(M) :g^{-1}U\in q\}$ \cite{Newelski-stable-groups}. Thus, if $\Delta$ is a finite family of formulae, in order to restrict the multiplication to $S_{G,\Delta}(M)$, the set of $\Delta$-types on $G$, we'll need to consider invariant families of formulae:

\begin{definition}
Let $\Delta\subseteq L$ be a finite set of formulae. We'll say that $\Delta$ is ($G$-)invariant if the family of subsets of $G$ definable by instances of formulae from $\Delta$ is invariant under left and right translation in $G$.
\end{definition}

From now on, unless stated otherwise, we'll assume that $\Delta$ is a finite set of invariant formulae. For $\Delta_1\subseteq \Delta_2$ let \[r^{\Delta_2}_{\Delta_1}:S_{G,\Delta_2}(M)\to S_{G,\Delta_1}(M)\] be the restriction map. These are semigroup homomorphisms. Thus \[S_G(M)=\varprojlim_{\Delta}S_{G,\Delta}(M).\]
In \cite{Newelski-stable-groups} Newelski shows that $S_{G,\Delta}(M)$ may be interpreted in $M^{eq}$ as an $\infty$-definable semigroup. Our aim is to show that these $\infty$-definable semigroups are in fact an intersection of definable ones and as a consequence that $S_G(M)$ is an inverse limit of definable semigroups of $M^{eq}$.

\subsection{$S_{G,\Delta}$ is an intersection of definable semigroups}

Let $\varphi(x,y)$ be a $G$-invariant formula. 
The proof that  $S_{G,\varphi}(M)$ is interpretable as an $\infty$-definable semigroup in $M^{eq}$ is given by Newelski in \cite{Newelski-stable-groups}. We'll show that it may be given as an intersection of definable semigroups.

\begin{proposition}\cite{Pillay}\label{P:def-in-stab}
There exists $n\in\mathbb{N}$ and a formula $d_\varphi(y,u)$ such that for every $p\in S_{G,\varphi}(M)$ there exists a tuple $c_p\subseteq G$ such that
\[d_\varphi(y,c_p)=(d_px)\varphi (x,y).\]
Moreover, $d_\varphi$ may be chosen to be a positive boolean combination of $\varphi$-formulae.
\end{proposition}

Let $E_{d_\varphi}$ be the equivalence relation defined by 
\[c_1E_{d_\varphi} c_2 \Longleftrightarrow \forall y(d_\varphi (y,c_1)\leftrightarrow d_\varphi (y,c_2)).\]
Set $Z_{d_\varphi}:=\nicefrac{M}{E_{d_\varphi}}$, it is the sort of canonical parameters for a potential $\varphi$-definition.

\begin{remark}
We may assume that $c_p$ is the canonical parameter for $d_\varphi (M,c_p)$, namely, that it lies in $Z_{d_\varphi}$. Just replace the formula $d_\varphi(y,u)$ with the formula 
\[\psi(y,v)=\forall u \left((\pi(u)=v)\to d_\varphi(y,u)\right),\]
where $v$ lies in the sort $\nicefrac{M}{E_{d_\varphi}}$ and $\pi:M\to\nicefrac{M}{E_{d_\varphi}}$.
\end{remark}

Each element $c\in Z_{d_\varphi}$ corresponds to a complete (but not necessarily consistent) set of $\varphi$-formulae:

\[p^0_c:=\{\varphi(x,a): a\in M \text{ and } \models d_\varphi(a,c)\}\cup\] \[\{\neg\varphi(x,a): a\in M \text{ and } \not\models d_\varphi(a,c)\}.\]

\begin{remark}
Notice that $p^0_c$ may not be closed under equivalence of formulae, but the set of canonical parameters $c\in Z_{d_\varphi}$ such that $p^0_c$ is closed under equivalence of formulae is the definable set:
\[\{c\in Z_{d_\varphi}: \forall t_1 \forall t_2 \left(\varphi(x,t_1)\equiv \varphi(x,t_2)\to (d_\varphi(t_1,c)\leftrightarrow d_\varphi(t_2,c)\right)\}.\]
Thus we may assume that we only deal with sets $p^0_c$ which are closed under equivalence of formulae.
\end{remark}

The set of $c\in Z_{d_\varphi}$ such that $p^0_c$ is $k$-consistent is definable: \[Z_{d_\varphi}^k=\{c\in Z_{d_\varphi} : p^0_c \text{ is k-consistent}\}.\]
Define \[Z=\bigcap_{k< \omega} Z_{d_\varphi}^k.\]
There is a bijection ($p\mapsto c_p$) between $S_{G,\varphi}(M)$ and $Z$.

The following is a trivial consequence of \Cref{P:def-in-stab}:
\begin{lemma}
There exists a formula $\Phi(u,v,y)$ with $u,v$ in the sort $Z_{d_\varphi}$, such that \[\Phi(c_p,c_q,a) \Leftrightarrow \varphi(x,a)\in p\cdot q.\] Moreover, $\Phi$ is a positive boolean combination of $d_\varphi$-formulae (and so of $\varphi$-formulae as well).
\end{lemma}
\begin{proof}
Since $\varphi$ is $G$-invariant, for simplicity we'll assume that $\varphi(x,y)$ is in fact of the form $\varphi(l\cdot x\cdot r,y)$. 
Let $c_p,c_q\subseteq G$ be tuples whose images in $Z_{d_\varphi}$ correspond to the $\varphi$-types $p,q\in S_{G,\varphi}(M)$, respectively.

Remembering that $u=(u_{ij})_{1\leq i,j\leq n}$ is a tuple of variables, we may write \[d_\varphi(l,r,y,u)=\bigvee_{i<n}\bigwedge_{j<n}\varphi(l\cdot u_{ij}\cdot r,y).\]

Since \[d_q(\varphi(b\cdot x\cdot c,a))=\{g\in G(M):\varphi((b\cdot g)\cdot x\cdot c,a)\in q)\}\] \[=\{g\in G(M) :\models d_\varphi(b\cdot g,c,a,c_q)\}\]
and
\[d_\varphi(b\cdot g,c,a,c_q)=\bigvee_{i<n}\bigwedge_{j<n}\varphi(b\cdot g\cdot ((c_q)_{ij}\cdot c),a),\]

we get that \[\varphi(b\cdot x\cdot c,a)\in p\cdot q \Longleftrightarrow \models \bigvee_{i<n}\bigwedge_{j<n}d_\varphi(b,((c_q)_{ij}\cdot c),a,c_p).\]

\end{proof}

Using this we define a partial binary operation on $Z_{d_\varphi}$:
\begin{definition}
For $c_1,c_2,d\in Z_{d_\varphi}$, we'll say that $c_1\cdot c_2=d$ if $d$ is the unique element of $Z_{d_\varphi}$ that satisfies \[ \models d_\varphi(a,d)\Longleftrightarrow \models \Phi(c_1,c_2,a).\]
for all $a\in M$.
\end{definition}

By compactness, there exists $k\in\mathbb{N}$ such that for all $c_1,c_2\in Z^k_{d_\varphi}$ there exists a unique $d\in Z_{d_\varphi}$ such that $c_1\cdot c_2=d$. For simplicity, we'll assume that this happens for $Z^1_{d_\varphi}$.

\begin{theorem}
$Z$ is contained in a definable semigroup extending the multiplication on $Z$.
\end{theorem}
\begin{proof}
By compactness there exists $k\in\mathbb{N}$ such that the multiplication is associative on $Z^k_{d_\varphi}$ and the product of two elements of $Z^k_{d_\varphi}$ is in $Z^1_{d_\varphi}$. For simplicity, let's assume that this happens for $Z^2_{d_\varphi}$. 
\begin{claim}
If $c_p\in Z$ and $c\in Z^2_{d_\varphi}$ then $c_p\cdot c\in Z^2_{d_\varphi}.$
\end{claim}
Let $U_1,U_2\in p_{c_p\cdot c}$, hence \[\{g\in G(M):g^{-1}U_1\in p^0_c\},\{g\in G(M):g^{-1}U_2\in p^0_c\}\in p.\] Since $p$ is consistent, there exists $g\in G(M)$ such that $g^{-1}U_1,g^{-1}U_2\in p^0_c$. Since $c\in Z^2_{d_\varphi}$,  $p^0_c$ is $2$-consistent. Thus, the claim follows.

Define \[\widehat{Z^2_{d_\varphi}}=\{c\in Z^2_{d_\varphi}: c\cdot Z^2_{d_\varphi}\subseteq Z^2_{d_\varphi}\}.\]
$\widehat{Z^2_{d_\varphi}}$ is the desired definable semigroup.
\end{proof}

\begin{corollary}
$Z=S_{G,\varphi}(M)$ is an intersection of definable semigroups.
\end{corollary}

Looking even closer at the above proof we may show that $S_G(M)$ is an inverse limit of definable semigroups:

Assume that $\Delta_2=\{\varphi_1,\varphi_2\}$ and $\Delta_1=\{\varphi_1\}$. In the above notations:
\[Z_{\Delta_2}=Z_{d_{\varphi_1}}\times Z_{d_{\varphi_2}}\]
For $c=\langle c_1,c_2\rangle \in Z_{\Delta_2}$ define 
\[p^0_c=p^0_{c_1} \cup p^0_{c_2}\] and then \[Z(\Delta_2)=\bigcap Z^k_{\Delta_2}\] similarly. 

For $c,c',d\in Z(\Delta_2)$ we'll say that $c\cdot c'=d$ if $d$ is the unique element $d\in Z_(\Delta_2)$ that satisfies:
\[c_1\cdot c'_1=d_1 \text{ and } c_2\cdot c'_2=d_2\]
As before, we assume that such a unique element exists already for any pair of elements in $Z^1_{\Delta_2}=Z^1_{\varphi_1}\times Z^1_{\varphi_2}$. The restriction maps $r^{\Delta_2}_{\Delta_1}:Z^1_{\Delta_2}\to Z^1_{\Delta_1}$ are definable homomorphisms.
Generally, for every $\Delta=\{\varphi_1,\cdots, \varphi_n\}$ and $i<\omega$ \[Z_\Delta^i=Z_{\varphi_1}^i\times\cdots \times Z_{\varphi_n}^i\] the multiplication is coordinate-wise. So the restriction commutes with the inclusion.
As a result,
\begin{theorem}\label{T:invlim}
$S_G(M)$ is an inverse limit of definable semigroups:
\[\varprojlim_{\Delta,i} Z^\Delta_i=\varprojlim_\Delta S_{G,\Delta}(M).\]
\end{theorem}

\subsection{The case where $S_G(M)$ is an inverse monoid}\label{ss:S_G inverse}
We would like to use the Theorems we proved in \Cref{S:big-wedge-defin} to improve the result in the situation where $S_G(M)$ is an inverse monoid. We'll first see that this situation might occur. Notice that the inverse operation $^{-1}$ on $S_G(M)$ is an involution.

\begin{proposition}\cite{Lawson}\label{P:semi-topo semigroup}
Let $S$ be a compact semitopological $*$-semigroup (a semigroup with involution) with a dense unit group $G$. Then the following are equivalent for any element $p\in S$:
\begin{enumerate}
\item $p=pp^*p$;
\item $p$ has a unique quasi-inverse;
\item $p$ has an quasi-inverse.
\end{enumerate}
\end{proposition}

\begin{remark}
In the situation of $G(M)\hookrightarrow S_G(M)$, the above Proposition can be proved directly using model theory and stabilizers.
\end{remark}

Translating the above result to our situation and using results in $1$-based groups (see \cite{Pillay}):
\begin{corollary}
The following are equivalent:
\begin{enumerate}
\item for every $p\in S_G(M)$, $p$ is the generic of a right coset of a connected $M$-$\infty$-definable subgroup of $G$;
\item for every $p\in S_G(M)$, $p\cdot p^{-1}\cdot p=p$;
\item $S_G(M)$ is an inverse monoid;
\item $S_G(M)$ is a regular monoid.
\end{enumerate}
Thus if $G$ is $1$-based then $S_G(M)$ is an inverse monoid.
\end{corollary}
\begin{proof}
$(2),(3)$ and $(4)$ are equivalent by Proposition \ref{P:semi-topo semigroup} and $(1)$ is equivalent to $(2)$ by \cite[Lemma 1.2]{Kowalski}
\end{proof}

With a little more work one may characterise when $S_G(M)$ is a Clifford Monoid.
\begin{definition}
A right-and-left coset of a subgroup $H$ is a right coset $Ha$ such that $aH=Ha$.
\end{definition}

\begin{proposition}\cite{Kowalski}
$p\in S_G(M)$ is a generic of a right-and-left coset of an $M$-$\infty$-definable connected subgroup of $G$ if and only if $p\cdot p\cdot p^{-1}=p$.
\end{proposition}

By using the following easy lemma
\begin{lemma}
Assuming $pp^{-1}p=p$, $$pp^{-1}=p^{-1}p \Leftrightarrow ppp^{-1}=p.$$
\end{lemma}

we get

\begin{proposition}
The following are equivalent:
\begin{enumerate}
\item every $p\in S_G(M)$ is the generic of a right-and-left coset of a connected $M$-$\infty$-definable subgroup of $G$;
\item $S_G(M)$ is a Clifford monoid.
\end{enumerate}
\end{proposition}

As a result, it may happen that $S_G(M)$ is an inverse (or Cliford) monoid. One may wonder if in these cases we may strengthen the result.

\begin{lemma}
If $S_G(M)$ is a Clifford Monoid then so is $S_{G,\Delta}(M)$. They same goes for inverse monoids.
\end{lemma}
\begin{proof}
In order to show that $S_{G,\Delta}(M)$ is a Clifford Monoid we must show that it is regular and that the idempotents are central.

Indeed this follows from the fact that the restriction maps are surjective homomorphisms and that if $q|_\Delta$ is an idempotent there exists an idempotent $p\in S_G(M)$ such that $p|_\Delta=q|_\Delta$ \cite{Newelski-stable-groups}.
\end{proof}

Assume that $\Delta$ is a finite invariant set of formulae. We'll show that if $S_G(M)$ is an inverse monoid then $S_{G,\Delta}(M)$ is an intersection of definable inverse monoids.

We recall some definition from \cite{Newelski-stable-groups}. Since $\Delta$ is invariant for $p\in S_{G,\Delta}(M)$ we have a map
\[d_p:Def_{G,\Delta}(M)\to Def_{G,\Delta}(M)\]
defined by
\[U\mapsto \{g\in G(M): g^{-1}U\in p\}.\]
Here $Def_{G,\Delta}(M)$ are the $\Delta$-$M$-definable subsets of $G(M)$.

Furthermore, for $p\in S_{G,\Delta}(M)$ define 
\[Ker(d_p)=\{U\in Def_{G,\Delta}(M) : d_p(U)=\emptyset\}.\]

\begin{lemma}\cite{Newelski-stable-groups}
Let $\Delta$ be a finite invariant set of formulae and $p\in S_{G,\Delta}(M)$ be an idempotent. Then
\[\{q\in S_{G,\Delta}(M):Ker(d_q)=Ker(d_p)\}=\{g\cdot p: g\in G(M)\}=G(M)p.\]
In particular it is definable (in $M^{eq}$).
\end{lemma}

\begin{corollary}
If $S_{G,\Delta}(M)$ is a regular semigroup then
\[S_{G,\Delta}(M)=\bigcup_{p \text{ idempotent}} G(M)p.\]
\end{corollary}
\begin{proof}
Let $q\in S_{G,\Delta}(M)$. By regularity there exists $\tilde{q}$ such that \[q=q\tilde{q}q \text { and } \tilde{q}=\tilde{q}q\tilde{q}.\]
$\tilde{q}q$ is the desired idempotent and $Ker(q)=Ker(\tilde{q}q).$
\end{proof}

Recall \Cref{T:Needed_for_1_based}. Notice that a semigroup $S$ fulfilling the requirements of the Theorem is an inverse monoid. It is obviously regular and the pseudo-inverse is unique since the idempotents commute. We get the following:

\begin{corollary}
If $S_G(M)$ is an inverse semigroup then $S_{G,\Delta}(M)$ is an intersection of definable inverse semigroups.
\end{corollary}
\begin{proof}
Since $S_G(M)$ is inverse so are the $S_{G,\Delta}(M)$. By \cite{Newelski-stable-groups} the unit group of $S_{G,\Delta}(M)$ is definable and by the previous corollary for every $p\in S_{G,\Delta}(M)$ there exists an idempotent $e$ and $g\in G(M)$ such that \[p=ge.\]
\end{proof}

\paragraph*{Acknowledgement}
I would like to thank my PhD advisor, Ehud Hrushovski for our discussions, his ideas and support, and his careful reading of previous drafts.

\bibliographystyle{plain}
\bibliography{semigroups-pub}

\end{document}